\documentclass[12pt]{amsart}
\usepackage{amsmath,amssymb,amsbsy,amsfonts,amsthm,latexsym,mathabx,
            amsopn,amstext,amsxtra,euscript,amscd,stmaryrd,mathrsfs,
            cite,array,mathtools,enumerate}

\usepackage{url}
\usepackage[colorlinks,linkcolor=blue,anchorcolor=blue,citecolor=blue,backref=page]{hyperref}

\usepackage[norefs,nocites]{refcheck}
\usepackage{color}
\usepackage{float}

\hypersetup{breaklinks=true}

\usepackage[english]{babel}

\usepackage{mathtools}
\usepackage{todonotes}

\usepackage{enumitem}

\usepackage{mathtools}
\usepackage{todonotes}
\usepackage{url}
\usepackage[colorlinks,linkcolor=blue,anchorcolor=blue,citecolor=blue,backref=page]{hyperref}

\begin{document}

\newtheorem{theorem}{Theorem}
\newtheorem{lemma}[theorem]{Lemma}
\newtheorem{claim}[theorem]{Claim}
\newtheorem{cor}[theorem]{Corollary}
\newtheorem{prop}[theorem]{Proposition}
\newtheorem{definition}{Definition}
\newtheorem{question}[theorem]{Open Question}
\newtheorem{example}[theorem]{Example}
\newtheorem{remark}[theorem]{Remark}

\numberwithin{equation}{section}
\numberwithin{theorem}{section}

 \newcommand{\F}{\mathbb{F}}
\newcommand{\K}{\mathbb{K}}
\newcommand{\D}[1]{D\(#1\)}
\def\scr{\scriptstyle}
\def\\{\cr}
\def\({\left(}
\def\){\right)}
\def\[{\left[}
\def\]{\right]}
\def\<{\langle}
\def\>{\rangle}
\def\fl#1{\left\lfloor#1\right\rfloor}
\def\rf#1{\left\lceil#1\right\rceil}
\def\le{\leqslant}
\def\ge{\geqslant}
\def\eps{\varepsilon}
\def\mand{\qquad\mbox{and}\qquad}

\def\vec#1{\mathbf{#1}}

\def\bl#1{\begin{color}{blue}#1\end{color}} 

\newcommand{\QQ}{\mathbb{Q}}

\newcommand{\C}{\mathbb{C}}
\newcommand{\Fq}{\mathbb{F}_q}
\newcommand{\Fp}{\mathbb{F}_p}
\newcommand{\Disc}[1]{\mathrm{Disc}\(#1\)}
\newcommand{\Res}[1]{\mathrm{Res}\(#1\)}
\newcommand{\ord}{\mathrm{ord}}

\newcommand{\Z}{\mathbb{Z}}
\renewcommand{\L}{\mathbb{L}}

\newcommand{\Norm}{\mathrm{Norm}}

\def\cA{{\mathcal A}}
\def\cB{{\mathcal B}}
\def\cC{{\mathcal C}}
\def\cD{{\mathcal D}}
\def\cE{{\mathcal E}}
\def\cF{{\mathcal F}}
\def\cG{{\mathcal G}}
\def\cH{{\mathcal H}}
\def\cI{{\mathcal I}}
\def\cJ{{\mathcal J}}
\def\cK{{\mathcal K}}
\def\cL{{\mathcal L}}
\def\cM{{\mathcal M}}
\def\cN{{\mathcal N}}
\def\cO{{\mathcal O}}
\def\cP{{\mathcal P}}
\def\cQ{{\mathcal Q}}
\def\cR{{\mathcal R}}
\def\cS{{\mathcal S}}
\def\cT{{\mathcal T}}
\def\cU{{\mathcal U}}
\def\cV{{\mathcal V}}
\def\cW{{\mathcal W}}
\def\cX{{\mathcal X}}
\def\cY{{\mathcal Y}}
\def\cZ{{\mathcal Z}}

\def\fra{{\mathfrak a}} 
\def\frb{{\mathfrak b}}
\def\frc{{\mathfrak c}}
\def\frd{{\mathfrak d}}
\def\fre{{\mathfrak e}}
\def\frf{{\mathfrak f}}
\def\frg{{\mathfrak g}}
\def\frh{{\mathfrak h}}
\def\fri{{\mathfrak i}}
\def\frj{{\mathfrak j}}
\def\frk{{\mathfrak k}}
\def\frl{{\mathfrak l}}
\def\frm{{\mathfrak m}}
\def\frn{{\mathfrak n}}
\def\fro{{\mathfrak o}}
\def\frp{{\mathfrak p}}
\def\frq{{\mathfrak q}}
\def\frr{{\mathfrak r}}
\def\frs{{\mathfrak s}}
\def\frt{{\mathfrak t}}
\def\fru{{\mathfrak u}}
\def\frv{{\mathfrak v}}
\def\frw{{\mathfrak w}}
\def\frx{{\mathfrak x}}
\def\fry{{\mathfrak y}}
\def\frz{{\mathfrak z}}

\def\ov\QQ{\overline{\QQ}}
\def \brho{\boldsymbol{\rho}}

\def \pf {\mathfrak p}

\def \Prob{{\mathrm {}}}
\def\e{\mathbf{e}}
\def\ep{{\mathbf{\,e}}_p}
\def\epp{{\mathbf{\,e}}_{p^2}}
\def\em{{\mathbf{\,e}}_m}

\newcommand{\sR}{\ensuremath{\mathscr{R}}}
\newcommand{\sDI}{\ensuremath{\mathscr{DI}}}
\newcommand{\DI}{\ensuremath{\mathrm{DI}}}

\newcommand{\Orb}[1]{\mathrm{Orb}\(#1\)}
\newcommand{\aOrb}[1]{\overline{\mathrm{Orb}}\(#1\)}

\def \Nm{{\mathrm{Nm}}}
\def \Gal{{\mathrm{Gal}}}

\newenvironment{notation}[0]{%
  \begin{list}%
    {}%
    {\setlength{\itemindent}{0pt}
     \setlength{\labelwidth}{1\parindent}
     \setlength{\labelsep}{\parindent}
     \setlength{\leftmargin}{2\parindent}
     \setlength{\itemsep}{0pt}
     }%
   }%
  {\end{list}}

\definecolor{dgreen}{rgb}{0.,0.6,0.}
\def\tgreen#1{\begin{color}{dgreen}{\it{#1}}\end{color}}
\def\tblue#1{\begin{color}{blue}{\it{#1}}\end{color}}
\def\tred#1{\begin{color}{red}#1\end{color}}
\def\tmagenta#1{\begin{color}{magenta}{\it{#1}}\end{color}}
\def\tNavyBlue#1{\begin{color}{NavyBlue}{\it{#1}}\end{color}}
\def\tMaroon#1{\begin{color}{Maroon}{\it{#1}}\end{color}}

\title[Dynamical irreducibility modulo primes]{Dynamical irreducibility of polynomials  modulo primes}

  \author[L. M{\'e}rai]{L{\'a}szl{\'o} M{\'e}rai}
 \address{L.M.: Johann Radon Institute for Computational and Applied Mathematics, 
 Austrian Academy of Sciences,  Altenberger Stra\ss e 69, A-4040 Linz, Austria}
  \email{laszlo.merai@oeaw.ac.at}
 \author[A.~Ostafe]{Alina Ostafe}
 \address{A.O.: School of Mathematics and Statistics, University of New South Wales.
 Sydney, NSW 2052, Australia}
 \email{alina.ostafe@unsw.edu.au}
 \author[I.~E.~Shparlinski]{Igor E. Shparlinski}
 \address{I.E.S.: School of Mathematics and Statistics, University of New South Wales.
 Sydney, NSW 2052, Australia}
 \email{igor.shparlinski@unsw.edu.au}
 
 \begin{abstract}  For a class of polynomials $f \in \Z[X]$, 
 which in particular includes all quadratic polynomials, and also trinomials of some special form, 
 we show that, under some natural conditions (necessary for quadratic polynomials),  the set of primes $p$ such that  all iterations of $f$ are irreducible 
 modulo $p$ is of relative density zero. Furthermore, we give an explicit bound on the rate of the decay of the density of such 
 primes in an interval $[1, Q]$ as $Q \to \infty$. 
 For this class of polynomials this gives a  more precise 
 version of a recent  result of A.~Ferraguti (2018), which   applies to arbitrary polynomials but requires a certain 
 assumption about their Galois group. Furthermore, under the Generalised Riemann Hypothesis we obtain a stronger bound 
 on this density. 
 \end{abstract}

\pagenumbering{arabic}

\maketitle

\section{Introduction}
\subsection{Motivation}

For a polynomial $f\in \K[X]$ over a field $\K$ we define
the sequence of polynomials:
$$
  f^{(0)}(X)  = X, \qquad f^{(n)}(X)  = f\(f^{(n-1)}(X)\), \quad
  n =1, 2, \ldots\,.
$$
The polynomial $f^{(n)}$ is called the $n$-th iterate of the polynomial $f$.

Following the established terminology, see~\cite{Ali,AyMcQ,Jon1,JB,Gomez_Nicolas,GNOS}, one says that a polynomial $f\in \K[X]$
is {\it stable\/} if all iterates $ f^{(n)}(X)$, $n =1,2 , \ldots$,  are irreducible over $\K$. However, 
 we prefer to use the more informative terminology introduced by Heath-Brown and Micheli~\cite{H-BM} and 
 instead we  call such polynomials {\it dynamically irreducible\/}.

For a polynomial $f\in\QQ[X]$ and a prime $p$ we define $f_p\in\Fp[X]$ to be the {\it reduction\/} of $f$ modulo $p$. 
In this paper we consider the following question, see~\cite[Question~19.12]{BdMIJMST}.

\begin{question}
 Let $f\in\QQ[X]$ be a dynamically irreducible polynomial of degree $d\geq 2$. Is it true that the set of primes
\begin{equation}
\label{eq:stable p}
 \{p: f_p \text{ is dynamically irreducible over } \Fp \}
\end{equation}
 is a finite set?
\end{question}
For example, Jones~\cite[Conjecture~6.3]{Jon2} has conjectured that $x^2+1$ is dynamically irreducible over $\Fp$ if and only if $p=3$. Ferraguti~\cite[Theorem~2.3]{Fer} has shown that if the size of the Galois group $\Gal\(f^{(n)}\)$  of $f^{(n)}$ is asymptotically close to its largest possible value then the set of  primes~\eqref{eq:stable p} has density zero. It is natural to assume that this condition on 
the size of $\Gal\(f^{(n)}\)$  is generically satisfied, however it may be difficult to verify it for concrete polynomials or find examples of such polynomials. 

Here we consider   a special class of polynomials which includes trinomials of the form $f(X)=aX^d+bX^{d-1}+c\in \Z[X]$ of even degree, and hence all quadratic polynomials.   For these polynomials, we prove such a zero-density result for the set of primes~\eqref{eq:stable p}, which holds under some  mild assumptions, that are also easily verifiable from the initial data.
Moreover, combining 
 \begin{itemize}
\item some effective results from Diophantine geometry~\cite{BEG}, 
\item the square-sieve of Heath-Brown~\cite{H-B}, 
\item a slightly refined bound of character sums over almost-primes from~\cite{KonShp}, 
\end{itemize}
we obtain an explicit saving in our density estimate. 

Furthermore, assuming the Generalised Riemann Hypothesis (GRH), we obtain a stronger bound. 

We believe these techniques have never been used before in this combination  and for similar purposes. 
Hence we  expect that this approach may find several other applications. 

\subsection{Main results}
\label{sec:res}
Clearly, it is enough to consider the distribution of primes for which $f_p$ is dynamically irreducible in dyadic intervals 
of the form $[Q,2Q]$. 
Thus, given a polynomial $f(X) \in \Z[X]$ we define
\begin{equation}
\begin{split}
\label{eq:PfQ}
 P_f(Q)=\#\{p\in [Q&,   2Q] \cap \cP:\\
 & ~f_p \text{ is dynamically irreducible over } \Fp \},
 \end{split} 
\end{equation}
where $\cP$ denotes the set of primes.

\begin{theorem}
\label{thm:lin sqr}
Let $f(X) \in \Z[X]$ be  
such that the derivative $f'(X)$ is of the form
\begin{equation}\label{eq:f'shape}
f'(X)= g(X)^2 (aX+b),  \qquad  g(X) \in \Z[X], \ a,b \in \Z, \ a \ne 0.
\end{equation}
Assume that   
$\gamma=-b/a$ is not a pre-periodic point of $f$. Then  one has
$$
 P_f(Q)  \le \frac{(\log \log\log\log Q)^{2+o(1)}} {\log\log\log Q}  \cdot \frac{Q }{\log Q}, \qquad  \text{as}\  Q\to \infty.
$$
\end{theorem}

Obviously all quadratic polynomials have their derivatives of the form required in Theorem~\ref{thm:lin sqr}. 

We also note that for quadratic polynomials the  condition of $\gamma$ not to be a pre-periodic point of $f$
in Theorem~\ref{thm:lin sqr} is necessary, as otherwise 
using~\cite[Lemma~2.5]{JB} (see also~\cite[Lemma~2.5']{JB-Err}) one can produce all primes in an arithmetic progression 
contained in the set~\eqref{eq:stable p}. For example, for the polynomial $f(X)=(X-2)^2+2$ with $\gamma=2$, the reduction $f_p$ is dynamically irreducible for all primes $p\equiv 5 \mod 8$. 
This example also shows that the condition is needed for higher degree polynomials as well, for example, consider $g(X)=f^{(k)}(X)$ for some $k\ge 2$, which is a polynomial of degree $2^k$ and by the above $g$ is also dynamically irreducible for all primes $p\equiv 5 \mod 8$.

 One can give similar examples for odd degrees as well. For example, 
 it follows from~\cite[Theorem~3.75]{LidlNiederreiter} that for the polynomial $f(X)=(X+2)^3-2$, $f_p$ is dynamically irreducible if 
 \begin{equation}\label{eq:cubic}
 \text{2 is a cubic non-residue modulo $p$ and } p\equiv 4, 7 \mod 9. 
 \end{equation}
The set of such primes $p$ that 2 is a cubic \textit{residue} modulo $p$, or equivalently, of  the form $p=x^2+27y^2$ with integers $x$ and $y$, by~\cite[Theorem~9.12]{cox}, is of Dirichlet density at most $1/h(-108)<1/3$ (see, for example,~\cite[Equation~(2.14)]{cox}).  Thus the set of primes~\eqref{eq:cubic}  is of positive Dirichlet density.

We now exhibit a
 larger class of polynomials of higher degree to which Theorem~\ref{thm:lin sqr} applies.

\begin{cor}
\label{cor:trinom}
Let $f(X)=r(X-u)^d+s(X-u)^{d-1}+t\in \Z[X]$ with some $r,s,t,u \in \Z$, $r \ne 0$, be 
such that  $d$ is even  
and 
$$\gamma=u-\frac{(d-1)s}{dr} 
$$
is not a pre-periodic point of $f$. Then
 $$
 P_f(Q)  \le \frac{(\log \log\log\log Q)^{2+o(1)}} {\log\log\log Q}  \cdot \frac{Q }{\log Q},\qquad \text{as} \ Q\to \infty.
$$
\end{cor}

We now give conditional (on the GRH) estimates.

\begin{theorem}
\label{thm:lin sqr-GRH}
Let $f(X) \in \Z[X]$ be  
such that the derivative $f'(X)$ is of the form~\eqref{eq:f'shape}.
Assume that   
$\gamma=-b/a$ is not a pre-periodic point of $f$. 
Then, assuming the GRH, for $Q \ge 3$, 
$$
 P_f(Q)  = O\( \frac{Q }{\log Q \log\log Q}\),  
 $$
where the implied constant depends only on $f$. 
\end{theorem}

Accordingly, we also have:

\begin{cor}
\label{cor:trinom-GRH}
Let $f(X)=r(X-u)^d+s(X-u)^{d-1}+t\in \Z[X]$ with some $r,s,t,u \in \Z$, $r \ne 0$, be 
such that  $d$ is even  
and 
$$\gamma=u-\frac{(d-1)s}{dr} 
$$
is not a pre-periodic point of $f$. Then, assuming the GRH,   for $Q \ge 3$, 
 $$
 P_f(Q)   =O\( \frac{Q }{\log Q \log\log Q}\),   
 $$
where the implied constant depends only on $f$. 
\end{cor}

We note that $s=0$ is admissible in Corollaries~\ref{cor:trinom} and~\ref{cor:trinom-GRH} and thus they  apply to binomials
of the form $r(X-u)^d +t$, which (for $r=1$ and $u=0$) are commonly  studied in arithmetic dynamics. 
Specially, Corollary~\ref{cor:trinom} also answers  a weakened version of   a conjecture of Jones~\cite[Conjecture~6.3]{Jon2}.

\section{Preliminaries}

 \subsection{Notation, general  conventions and definitions}

Throughout the paper, $p$ always denotes a prime number.

For a prime $p$, $v_p : \QQ^*\to \Z$ represents the usual {\it $p$-adic valuation\/}, that is, for $a\in \Z\setminus \{0\}$, we let $v_p(a)=k$ if $p^k$ is the highest power of $p$ which divides $a$, and for $a,b\in \Z\setminus \{0\}$ we let $v_p(a/b)=v_p(a)-v_p(b)$.

We define the {\it Weil logarithmic height\/} of $a/b\in\QQ$ as $$h(a/b)=\max \{ \log |a|,\log|b|\},$$
with the convention $h(0)=0$.

We use the Landau symbol $O$ and the Vinogradov symbol $\ll$. Recall that the
assertions $U=O(V)$ and $U \ll V$  are both equivalent to the inequality $|U|\le cV$ with some absolute constant $c>0$.
To emphasize the dependence of the implied constant $c$ on some parameter (or a list of parameters) $\rho$, we write $U=O_{\rho}(V)$
or $U \ll_{\rho} V$.

\subsection{Basic properties of resultants}
\label{sec:resultants}
Here we recall the following well known properties of resultants of polynomials, see~\cite{GaGe},
that hold over any field $\K$.

\begin{lemma}
\label{lem:res}
Let  $f,g\in\K[X]$ be polynomials of degrees $d\ge 1$ and $e\ge 1$, 
respectively, and let $h\in\K[X]$. 
Denote by $\beta_1,\ldots,\beta_e$ the roots of $g$ in an extension field. 
Then we have:
\begin{enumerate}
\renewcommand{\labelenumi}{(\roman{enumi})}
\item \quad  $\Res{f,g}=(-1)^{de}g_e^d\prod_{i=1}^e f(\beta_i)$,  
\item \quad  $\Res{fg,h}=\Res{f,h}\Res{g,h}$,
\end{enumerate}
where $g_e$ is the leading coefficient of $g$.
\end{lemma}

\subsection{Dynamically irreducible polynomials}

The following result gives a necessary condition that a polynomial is dynamically irreducible over a finite field of 
odd characteristic~\cite[Corollary 3.3]{GNOS}.

\begin{lemma}
\label{lemma:stab}
Let $q$ be an odd prime power, and let $f\in\F_q[X]$ be a dynamically irreducible polynomial 
 of degree $d\geq  2$ with leading coefficient $f_d$, nonconstant derivative 
 $f'$, $\deg f'=k\le d-1$. 
  Then the following properties hold: 
 \begin{enumerate}
 \item if $d$ is even, then $\Disc{f}$ and $f_d^k \Res{f^{(n)},f'}$, $n\geq 2$, are nonsquares in $\Fq$,  
 \item if $d$ is odd, then $\Disc{f}$ and $(-1)^{\frac{d-1}{2}}f_d^{(n-1)k+1} \Res{f^{(n)},f'}$, $n\geq 2$, are squares in $\Fq$. 
   \end{enumerate}
\end{lemma}

We note that when $d=2$, then by~\cite[Lemma~2.5]{JB} (see also~\cite[Lemma~2.5']{JB-Err})  the condition of Lemma~\ref{lemma:stab} is also sufficient. 

\subsection{Jacobi symbol}

For $n\geq 3$, $\left(\frac{\cdot}{n}\right)$ denotes the {\it Jacobi symbol\/}, which is identical to the {\it Legendre symbol\/}  if $n$ is prime. We recall the following well-known properties, see~\cite[Section~3.5]{IwKow}.

\begin{lemma}\label{lemma:reciprocity}
 For odd integers $m,n\geq 3$ we have
$$
  \left(\frac{m}{n}\right)=(-1)^{\frac{m-1}{2}\frac{n-1}{2}}\left(\frac{n}{m}\right)
  \mand
  \left(\frac{2}{n}\right)=(-1)^{\frac{n^2-1}{8}}.
  $$
\end{lemma}

\subsection{On some character sums over almost-primes}\label{subsect:char-sum}

For $\eta>0$, let $\cP(\eta,M)$ denote the set of positive integers $m\leq M$ which do not have prime divisors $p\leq M^\eta$.
It is well known that for for all positive $\eta <1$ and $M\geq 2$  one has
\begin{equation}
\label{eq:P bound}
\# \cP(\eta,M)\ll\frac{M}{\eta \log M}, 
\end{equation}
see~\cite[Part~III, Theorem~6.4 and Equation~(6.23)]{Ten}. 

One important tool in the proof of Theorem~\ref{thm:lin sqr} is the following result which is a slightly more
precise form of~\cite[Corollary~10]{KonShp}. Namely, it is easy to trace the dependence of $\eta$
in the second term of the bound of~\cite[Corollary~10]{KonShp} and see that the term $O_\eta\(M^{1-\eta}\)$
can be refined as $O\(\eta^{-2} M^{1-\eta}\)$.  More precisely, we have: 

\begin{lemma}\label{lemma:char sum}
For any $\varepsilon>0$ there exists some $\eta_0>0$ such that for any positive $\eta<\eta_0$, integer $M \geq q^{1/3+\varepsilon}$, where $q\geq 2$ is not a perfect square, we have
$$
\left|\sum_{m \in\cP(\eta,M)}\left(\frac{m}{q}\right) \right|\ll   \eta^{\eta^{-1/2}/4-1} \frac{M}{\log M} + \eta^{-2} M^{1-\eta}. 
$$
\end{lemma}

Furthermore, under the GRH we have a rather strong bound for sums over primes, 
see~\cite[Equation~(13.21)]{MonVau}

\begin{lemma}\label{lemma:char sum-GRH}
For any   positive integers $q$ and $M$, where $q\geq 2$ is not a perfect square, we have
$$
\left|\sum_{p \le M}\left(\frac{p}{q}\right) \right|   \ll M^{1/2} \log (qM). 
$$
\end{lemma}

\subsection{Integer solutions to hyperelliptic equations}
We also need the following effective result of B{\'e}rczes,   Evertse and Gy\H{o}ry~\cite[Theorem~2.2]{BEG}, which bounds the height  of $\cS$-integer solutions to a hyperelliptic equation. We present it in the form needed for the proof of Theorem~\ref{thm:lin sqr}.
 
Let $\cS$ be a finite set of primes of cardinality $s=\#\cS$ and define
$\Z_\cS$ to be the ring of $\cS$-integers, that is, the set of rational numbers $r$  with $v_p(r)\ge 0$ for any $p\not\in \cS$. 
Put
$$
Q_\cS=\prod_{p\in \cS} p.
$$

\begin{lemma}\label{lemma:BEG}
Let $f\in \Z_\cS[X]$ be a polynomial of degree $d\geq 3$ without multiple zeros, and let $b\in\Z_\cS$ be a nonzero $\cS$-integer. 
If $x,y\in \Z_\cS$ are solutions to the equation
$$
f(x)=by^2,
$$
then
$$
h(x),h(y)\leq (4ds)^{212d^4 s} Q_S^{20 d^3} \exp(O_{f}(h(b))).
$$
\end{lemma}

\subsection{On the height of some iterates and resultants}

We need the following simple estimates on the height of some iterates and  resultants:

\begin{lemma}
\label{lem:iter height}
Let $f\in\QQ[X]$ be a polynomial of degree $d\ge 1$ and let $\gamma\in\ov\QQ$. 
Then, there exists a constant $C_f$ depending only on $f$ such that for any $n\ge 1$ we have
$$
h\(f^{(n)}(\gamma)\) \ge d^{n}(h(\gamma)-C_f).
$$
\end{lemma}
\begin{proof}
The proof follows inductively applying~\cite[Theorem 3.11]{Silv}, this inequality is also given in~\cite[Equation~(3.8)]{Silv}. 
\end{proof}

We remark in Lemma~\ref{lem:iter height} we do not insist that $\gamma$ is  pre-periodic. Indeed, if it is, then $h(\gamma)$ 
is bounded and adjusting $C_f$ we can make  the result to be trivially correct.   

\begin{lemma}
\label{lem:res height}
Let $f\in\QQ[X]$ be a polynomial of degree $d\ge 1$. Then, for any $n\ge 1$, we have
$$
h\(\Res{f^{(n)},f'}\)=O_f\(d^n\).
$$
\end{lemma}
\begin{proof}
Let $f_d$ be the leading coefficient of $f$ and $\gamma_1,\ldots,\gamma_{d-1}$ be the roots of the derivative $f'$. Then $\Res{f^{(n)},f'}$ is defined by
$$
\Res{f^{(n)},f'}=(-1)^{d^n(d-1)}(df_d)^{d^n}\prod_{i=1}^{d-1}f^{(n)}(\gamma_i).
$$
We have
\begin{equation}
\label{eq:1}
h\((df_d)^{d^n}\)\le d^n(\log|d|+h(f_d))\ll_fd^n.
\end{equation}
Applying~\cite[Theorem 3.11]{Silv}, we also have
\begin{equation}
\label{eq:2}
h\(f^{(n)}(\gamma_i)\)\ll_f d^n.
\end{equation}
Putting~\eqref{eq:1} and~\eqref{eq:2} together, we conclude the proof.
\end{proof}

\section{Proof of Theorem~\ref{thm:lin sqr}}

\subsection{An application of  the square-sieve} 
We can assume that $f$ is dynamically irreducible over $\QQ$ as otherwise its reduction $f_p$ can be dynamically irreducible  for at most just finitely many primes $p$. 

Let $d = \deg f$. We can assume that $Q$ is large enough, thus
$f$ and $f'$ are of degrees $d$ and $d-1$, 
respectively, modulo  any prime $p\in[Q,2Q]$. 
From the shape~\eqref{eq:f'shape} of   $f'$  we see that $d-1$ is odd (and thus $d$ is even).

Let $\varepsilon>0$. All of the constants in this proof may depend on $\varepsilon$ and $f$.

Put
\begin{equation}
\label{eq:N t}
N = c_1 \log \log Q \mand t= c_2 \log\log\log Q
\end{equation}
with some sufficiently small constants $c_1, c_2>0$ fixed later.

Write 
$$
f_d \cdot \Res{f^{(n)},f'}=2^{\nu_n}u_n, \quad v_2(u_n)=0, \quad n\geq 2,
$$
where $f_d$ is the leading coefficient of $f$.  

By the Dirichlet principle there is a set $\cN\subseteq [N,N+t]$ of size
$$\#\cN\geq \frac{1}{4}t
$$ such that for all  $r,s\in\cN$ we have
$$
u_r\equiv u_s \mod 4  \mand  \nu_r\equiv \nu_s \mod 2.
$$
Therefore, since $u_n$, $n\ge 2$, are odd, we have
\begin{equation}
\label{eq:u nu}
u_r+u_s \equiv 2 \mod 4 \mand \nu_r+\nu_s\equiv 0 \mod 2.
\end{equation}
Using that for an odd $m$ we have $2 \mid m-1$ and $8 \mid m^2-1$, we conclude that 
\begin{equation}\label{eq:parity}
(-1)^{\frac{u_{r}+u_{s}-2}{2}\frac{m-1}{2}+(\nu_{r}+\nu_{s} ) \frac{m^2-1}{8}} =1.
 \end{equation}

Consider
$$
S=\sum_{p\in[Q,2Q]} \left|\sum_{n\in\cN} \left(\frac{f_d  \cdot \Res{f^{(n)},f'}}{p} \right)\right|^2.
$$

If $f_p$ is dynamically irreducible modulo $p$, then by Lemma~\ref{lemma:stab}, we have
\begin{align*}
\sum_{n\in\cN} \left(\frac{f_d \cdot \Res{f^{(n)},f'}}{p} \right)& =\sum_{n\in\cN} \left(\frac{f_d ^{d-1}\cdot \Res{f^{(n)},f'}}{p} \right)\\
& = -\#\cN\leq -\frac{t}{4},
\end{align*}
as $d-1$ is odd, 
thus
\begin{equation}
\label{eq:Pf S}
P_f(Q)\leq 16\frac{S}{t^2},
\end{equation}
where $P_f(Q)$ is defined by~\eqref{eq:PfQ}.

Let $\eta_0$ as in Lemma~\ref{lemma:char sum} and $\eta<\eta_0$ be chosen later. Then we extend the summation for integers $m\in\cP(\eta,2Q)$, where $\cP(\eta,2Q)$ is defined in Section~\ref{subsect:char-sum}, to obtain
$$
S\leq\sum_{m\in\cP(\eta,2Q)} \left|\sum_{n\in \cN}  \left(\frac{f_d  \cdot \Res{f^{(n)},f'}}{m} \right)\right|^2.
$$

By Lemma~\ref{lemma:reciprocity}, we have
$$
S\leq\sum_{m\in\cP(\eta,2Q)} \left|\sum_{n\in\cN}  (-1)^{\frac{u_n-1}{2}\frac{m-1}{2}+\nu_n \frac{m^2-1}{8}}\left(\frac{m}{u_n} \right)\right|^2.
$$ 
By opening the square and changing the order of summation, we see from~\eqref{eq:parity} that
\begin{align*}
S&\le 
\sum_{n_1,n_2\in\cN} \sum_{m\in\cP(\eta,2Q)}
(-1)^{\frac{u_{n_1}+u_{n_2}-2}{2}\frac{m-1}{2}+(\nu_{n_1}+\nu_{n_2} ) \frac{m^2-1}{8}}\left(\frac{m}{u_{n_1}u_{n_2}} \right)\\
&=\sum_{n_1,n_2\in\cN} \sum_{m\in\cP(\eta,2Q)}  \left(\frac{m}{u_{n_1}u_{n_2}}\right).
\end{align*}

Let $\cZ$ be the set of pairs $(n_1,n_2)\in\cN^2$ such that $u_{n_1}u_{n_2}$ is a square. For a pair $(n_1,n_2)\not\in\cZ$, we have by Lemma~\ref{lem:res height}, that 
\begin{align*}
|u_{n_1}&u_{n_2}|^{1/3+\varepsilon}\\
&\leq \exp \left((1/3+\varepsilon)\left(h\left(f_d\Res{f^{(n_1)},f' }\right)+h\left(f_d\Res{f^{(n_2)},f'} \right) \right) \right)\\
&\leq \exp\left(O_f(d^{N+t})\right) \leq Q
\end{align*}
if $c_1$ and $c_2$ are small enough. Then by Lemma~\ref{lemma:char sum}, and by \eqref{eq:P bound} we have
$$
S \ll   \# \cZ \frac{Q}{ \eta \log Q } + t^2\left(    \eta^{\eta^{-1/2}/4-1} \frac{Q}{\log Q} + \eta^{-2}Q^{1-\eta}  \right).
$$
Recalling~\eqref{eq:Pf S}, we now conclude
\begin{equation}\label{eq:almost-final}
P_f(Q)\ll  \frac{\# \cZ}{t^2} \frac{Q}{ \eta \log Q } +  \eta^{\eta^{-1/2}/4-1} \frac{Q}{\log Q} +\eta^{-2} Q^{1-\eta}.
\end{equation}

In the next section, we give a bound on $\#\cZ$.

\subsection{Perfect squares in denominators} \label{sec:nonsquares}

We show that $\cZ$ does not contain nontrivial (off-diagonal) pairs and hence 
\begin{equation}\label{eq:Z}
\# \cZ = \# \cN \le t. 
\end{equation}

Let $n_1\neq n_2$ be a pair of integers in $\cN$ such that $u_{n_1}u_{n_2}$ is a square. We can assume that $n_2>n_1$. Then, since 
$$
u_{n_1}u_{n_2}=2^{-\nu_{n_1}-\nu_{n_2}}  f_d ^2  \Res{f^{(n_1)},f'}\Res{f^{(n_2)},f'}
$$ 
and, by Lemma~\ref{lem:res},
$$
 \Res{f^{(n_1)},f'}\Res{f^{(n_2)},f'}=\Res{f^{(n_1)}f^{(n_2)},f'},
$$
we obtain, recalling~\eqref{eq:u nu}, that 
$$
\Res{f^{(n_1)}f^{(n_2)},f'}  
$$
is also a square.

Now, let 
$$
f'(X)= g(X)^2 (aX+b),  \qquad  g(X) \in \Z[X], \ a,b \in \Z, \ a \ne 0.
$$
Let $\beta_1, \ldots, \beta_m$ be the roots of $g$ (taken with multiplicities, that is, $m = (d-2)/2$). 

From here, using again Lemma~\ref{lem:res}, we obtain that 
\begin{equation*}
\begin{split}
\Res{f^{(n_1)}f^{(n_2)},f'}&=(-1)^{(d-1)(d^{n_1}+d^{n_2})}(df_d )^{d^{n_1}+d^{n_2}}\\
&\quad  \cdot \prod_{i=1}^m \(f^{(n_1)}(\beta_i)\)^2 \(f^{(n_2)}(\beta_i)\)^{2} f^{(n_1)}(\gamma)f^{(n_2)}(\gamma),
\end{split}
\end{equation*}
where $\gamma=-b/a\in\QQ$.

Now, since $d$ is even we have, that 
$$
f^{(n_1)}(\gamma)f^{(n_2)}(\gamma)
$$
is a square in $\QQ$. 

We let~$\cS$ be the set of primes, which consists of the prime divisors of $d$ and $a$. 
We thus have the equation
$$
\alpha f^{(n_2-n_1)}(\alpha) =\beta^2,
$$
where $\alpha=f^{(n_1)}(\gamma)$ and $\beta$ are $\cS$-integers in $\QQ$, and $\deg f^{(n_2-n_1)}\le d^t$.

Since $f$ is dynamically irreducible of degree at least two,  $f^{(n_2-n_1)}$ is irreducible and $X\nmid f^{(n_2-n_1)}$, and thus $Xf^{(n_2-n_1)}(X)$ is a polynomial of degree at least $3$ without multiple roots in $\ov\QQ$. We can apply now Lemma~\ref{lemma:BEG} with the polynomial $Xf^{(n_2-n_1)}(X)$. As the quantities $d,s$ and $Q_\cS$ depend only on $f$, we conclude that 
\begin{equation}
\label{eq:upper}
h(\alpha)\le \exp(O_f(d^{4t})).
\end{equation}
On the other hand, since $\gamma$ is not a pre-periodic point of the polynomial  $f$, there exists a positive integer $n_0$ depending only on $f$ such that $h\(f^{(n_0)}(\gamma)\)\ge C_f+1$, where $C_f$ is defined as in Lemma~\ref{lem:iter height}. Applying then Lemma~\ref{lem:iter height} we have
$$
h(\alpha)=h\(f^{(n_1-n_0)}\(f^{(n_0)}(\gamma)\)\)\ge d^{n_1-n_0}\(h\(f^{(n_0)}(\gamma)\)-C_f\) \ge d^{N-n_0}.
$$
We choose now a suitable constant $c_2$ in~\eqref{eq:N t}, depending only on $f$, and $Q$ large enough to obtain a contradiction with~\eqref{eq:upper}. We thus conclude that there is no nontrivial pair $n_1,n_2\in\cN$ such that  
$u_{n_1}u_{n_2}$
is a square in $\QQ$ which proves~\eqref{eq:Z}.

\subsection{Final optimisation} 
In order to conclude the proof, observe that~\eqref{eq:almost-final} and~\eqref{eq:Z} give 
$$
P_f(Q) \ll\frac{1}{t} \frac{Q}{ \eta \log Q } +   \eta^{\eta^{-1/2}/4-1} \frac{Q}{\log Q} +\eta^{-2}  Q^{1-\eta}. 
$$
Let us choose $\eta$ to satisfy 
$$
\eta^{\eta^{-1/2}/4}  = t^{-1}
$$
for which we derive 
\begin{equation}
\label{eq:pen-ult}
P_f(Q)\ll\frac{1}{t} \frac{Q}{ \eta \log Q}  +\eta^{-2}  Q^{1-\eta}.  
\end{equation}
Since for the above choice of $\eta$ we have 
$$
 \eta^{-1/2} \log \(\eta^{-1}\) = 4 \log t
 $$
 we conclude that  $\eta  = (\log t)^{-2+o(1)}$. 
 It is easy to check that with the choice of $t$ as in~\eqref{eq:N t}, 
 the second term in~\eqref{eq:pen-ult} never dominates and the result follows. 
 
 \section{Proof of Theorem~\ref{thm:lin sqr-GRH}}

 Since the proof is very similar to that of Theorem~\ref{thm:lin sqr} we only sketch the main steps.
 
 We now put
\begin{equation}
\label{eq:N t GRH}
N = c_3  \log Q \mand t= c_4  \log\log Q
\end{equation}
with some sufficiently small constants $c_3, c_4>0$ fixed later.

We recall the inequality~\eqref{eq:Pf S}, however this time we do not  expand the sum over primes in $S$ to the set $\cP(\eta,2Q)$.

As before, let $\cZ$ be the set of pairs $(n_1,n_2)\in\cN^2$ such that $u_{n_1}u_{n_2}$ is a square. For a pair $(n_1,n_2)\not\in\cZ$, again  by Lemma~\ref{lem:res height}, with the choice~\eqref{eq:N t GRH},   we conclude that 
$$
\log |u_{n_1}u_{n_2}| \ll_f d^{N+t} \leq Q^{1/3}
$$
if $c_3$ and $c_4$ are small enough. 
Hence, using  Lemma~\ref{lemma:char sum-GRH} instead of  Lemma~\ref{lemma:char sum}, we arrive to the 
following analogue of~\eqref{eq:almost-final}
\begin{equation}\label{eq:almost-final-GRH}
P_f(Q)\ll  \frac{\# \cZ}{t^2} \frac{Q}{  \log Q } +  Q^{5/6}.
\end{equation}

For $N$ and $t$ in~\eqref{eq:N t GRH} related similarly to those in~\eqref{eq:N t} we also have
the bound~\eqref{eq:Z}, which after substitution in~\eqref{eq:almost-final-GRH} gives
$$
P_f(Q)\ll  \frac{1}{t} \frac{Q}{  \log Q } +  Q^{5/6},
$$ 
and the result follows.

\section{Comments}

We remark that our method applies to any polynomial for which we can control the existence, or at 
least the frequency, of perfect squares in the products $\Res{f^{(n_1)},f'} \Res{f^{(n_2)},f'}$ with $N \le n_1< n_2 \le N+t$.  
We also remark that we have a lot of flexibility in selecting the interval 
$[N, N+t]$ from which $n_1$ and $n_2$ are chosen. Besides, we do not have to use all values from the set $\cN$ 
in the proofs of Theorems~\ref{thm:lin sqr}  and~\ref{thm:lin sqr-GRH}, 
but limit ourselves to certain (reasonably large) subset $\widetilde \cN \subseteq  \cN$ of integers with some desirable 
properties,  and then use  $\widetilde \cN$ in the argument.  

Unfortunately, despite the above flexibility of the method, 
besides the shifted trinomials of Corollaries~\ref{cor:trinom} and~\ref{cor:trinom-GRH},  we have not found any natural classes of polynomials for 
which this can be applied. Moreover, it is natural to try to extend Corollaries~\ref{cor:trinom} and~\ref{cor:trinom-GRH} to trinomials of odd degree. 
We note that in this case, the same approach as in Theorems~\ref{thm:lin sqr}  and~\ref{thm:lin sqr-GRH}   applies, but using Lemma~\ref{lemma:stab}~(2) instead of 
Lemma~\ref{lemma:stab}~(1). However the part about the square avoidance breaks down. 

Our main goal has been to establish an unconditional result, at least with respect to the polynomials we consider. However we observe that  under the celebrated
 {\it $ABC$-conjecture\/}
one can further extend the class of polynomials to which our method applies. To sketch this argument, for an integer $k \ne0$
we define  $\rho(k)$ as   the product of all distinct prime divisors of $k$, that is,
$$
\rho(k) = \prod_{p\mid k} p,
$$ 
which is also commonly called the {\it radical\/}  of $k$.
Next, we recall that Langevin~\cite{Lang} (see also~\cite[Theorem~5]{Gran}) has
shown that under $ABC$-conjecture, for any polynomial $g \in \Z[X]$ of degree $e \ge 2$, under some natural conditions, 
and any $m \in \Z$ we have 
$$
\rho\(g(m)\) \ge |m|^{e - 1 + o(1)}, \qquad |m| \to \infty. 
$$
 Hence if we write $|g(m)| = u v^2$ with a squarefree $u$ and an integer $v\ge 1$,  we see that 
$$
u v  \ge \rho\(g(m)\) \ge |m|^{e - 1 + o(1)}
$$
and using $v = \(|g(m)|/u\)^{1/2} \ll |m|^{e/2}  u^{-1/2}$, we derive
$$
u \ge  |m|^{e - 2 + o(1)}. 
$$

Suppose that all roots 
$$
|\gamma_1| > |\gamma_2|\ge \ldots \ge |\gamma_{d-1}|
$$
of $f'$ are integers and the largest root $\gamma_1$ is of  multiplicity one. 

Additionally, assume that the iterations  $f^{(n)}(\gamma_i)$ grow as expected, 
that is, doubly exponentially,  
$$
|f^{(n)}(\gamma_i)| = \exp\(\(1 + o(1)\)\vartheta_i^n \), \qquad i=1, \ldots, d-1, 
$$
for some  constants $\vartheta_i> 1$, and furthermore
$$
\vartheta_i < \vartheta_1^{(d-3)/d}, \qquad i =2, \ldots, d-1.
$$
In this case the squarefree part of $f^{(n_1)}(\gamma_1)$ is so large, 
namely, it is  at least
$$
  |f^{(n_1-1)}(\gamma_i) |^{d - 3 + o(1)} 
\ge \exp\(|\vartheta_1^{d^{n_1-1}(d - 3 + o(1))}\), 
$$
that the product  of other terms
$$
 \left| \prod_{i=2}^{d-1}f^{(n_1)}(\gamma_i) \prod_{i=1}^{d-1} f^{(n_2)}(\gamma_i)   \right| 
\le  \exp\((1+o(1))\(\sum_{i=1}^{d-1} \vartheta_2^{ d^{n_1}} + \vartheta_1^{d^{n_2}}\)\)
$$
is not large enough to complement it up to a square. 

Certainly this argument can be modified in several directions to cover many other scenarios and with an 
appropriate generalisation of the $ABC$-conjecture and the argument of~\cite{Lang,Gran} to number fields, 
see~\cite[Chapter~14]{BoGu}, it can work without the assumption of the integrality of the critical points. 
We do not pursue this venue here since, as we have mentioned, our goal is deriving 
unconditional  results for various classes of polynomials.

Finally, we remark on  obtaining analogues  of Theorems~\ref{thm:lin sqr}  and~\ref{thm:lin sqr-GRH}  when $f$ is an arbitrary polynomial with the property that its derivative is irreducible. In this case, following the same approach as in the proofs of Theorems~\ref{thm:lin sqr}  and~\ref{thm:lin sqr-GRH}, we get to the point when we have to discuss when $\Res{f^{(n_1)},f'} \Res{f^{(n_2)},f'}$ is a square. Using the irreducibility assumption of $f'$, we reduce this problem to analysing  when
$$
N_{\QQ(\gamma)/\QQ}\(f^{(n_1)}(\gamma)f^{(n_2)}(\gamma)\) = y^2
$$ 
for some  $y \in \QQ$, where $\gamma$ is one of the roots of $f'$ and $N_{\QQ(\gamma)/\QQ} : \QQ(\gamma) \to \QQ$ is the usual field norm map. To finalise our argument, under some natural assumptions on the polynomial $g\in\Z[X]$ (in our case $g=f^{(n_2-n_1)}$), we need an effective result for the height of $S$-integer solutions to the norm equation
$$
N_{\QQ(\gamma)/\QQ}(xg(x))=y^2,
$$
similar to those in~\cite{BEG}.

\section*{Acknowledgement}

The authors thank Andrea Ferraguti for feedback on an early version of the paper and pointing out an imprecision in   the initial statement of Theorem~\ref{thm:lin sqr} and supplying the example at the end of Section~\ref{sec:res}. 

During the preparation of this work, 
L.M. was supported by the Austrian Science Fund (FWF): Project  P31762,
A.~O. was supported by the
Australian Research Council (ARC): Grant DP180100201 and  
I.~S.   was  supported   by the Australian Research Council (ARC):  Grant DP170100786.

\end{document}